\documentclass{article}
\usepackage[utf8]{inputenc}
\usepackage{preamble} 
\graphicspath{{../figures/}} 

\usepackage[colorinlistoftodos]{todonotes}
\setuptodonotes{size=\scriptsize,backgroundcolor=red!15!white}

\newcommandx{\clay}[2][1=]{\todo[linecolor=cyan,backgroundcolor=cyan!25,bordercolor=cyan,#1]{\tiny Clay: #2}}
\newcommandx{\colin}[2][1=]{\todo[linecolor=green,backgroundcolor=green!25,bordercolor=green,#1]{\tiny Colin: #2}}

\title{A Gelfand Transform for Spinor Fields on Embedded Riemannian Manifolds}

\author{Colin Roberts\footnote{Ph.D. Candidate, Colorado State University, 1874 Campus Delivery, Fort Collins, CO 80523-1874}}
\affil{Colorado State University, Fort Collins, Colorado, 80523}

\begin{document}

\maketitle

\begin{abstract}
A classical result of Gelfand shows that the topologized spectrum of characters on commutative Banach algebra is homeomorphic to the underlying space. This fact is used in solving the Calder\'on problem in dimension 2 via the \emph{boundary control (BC) method}. To apply the BC method in dimension 3, the algebra of complex holomorphic functions can be replaced by the space of harmonic quaternion fields, but this space is no longer an algebra and is not commutative. Nonetheless, it has been shown that a suitable notion of a spectrum exists and in the case when the underlying space is convex, the spectrum is homeomorphic to the ball. Our goal is to generalize this result to more general manifolds in arbitrary dimension. To do so, we use Clifford algebras of multivector fields and the set of functions we care about is the space monogenic spinor fields. The spectrum consists of spinor valued functionals that respect the module and subalgebra structure of the space of monogenic spinor fields. We prove that for compact regions in $\R^n$, the spectrum is homeomorphic to the manifold itself. Furthermore, some essential lemmas and propositions are proven for arbitrary compact Riemannian $n$-manifolds. Finally, we end with a Stone--Weierstrass theorem which shows the algebra generated by the monogenic spinor fields on arbitrary compact $M$ with boundary is dense in the algebra of continuous spinor fields.
\end{abstract}

\section{Introduction}
This paper is motivated by the inverse tomography (or Calder\'on) problem for Riemannian manifolds \cite{calderon_inverse_1980,uhlmann_inverse_2014,lassas_determining_2001,lee_determining_1989,joshi_inverse_2005, belishev_dirichlet_2008,krupchyk_inverse_2011}. First, let us explain an important related problem called the Electrical Impedance Tomography (EIT) problem and note that it is equivalent to the classical Calder\'on problem in dimension $n=3$. Let $M$ be a body free of interior charges and made from a (possibly anisotropic) Ohmic material with conductivity $\gamma$. Applying a voltage $\phi$ on the boundary $\partial M$ induces a voltage $u$ in the interior and since the interior is free of charges, $u$ satisfies $\mathrm{div}(\gamma~ \mathrm{grad}(u))=0$. We are not allowed access to the interior, which leaves us solely with the ability to making measurements along the boundary $\partial M$. Given $\phi$, we can measure the corresponding current flux $\frac{\partial u}{\partial \normal}$ on $\partial M$ where $\normal$ is the outward normal field. Hence, we have a map $\Lambda$ called the voltage-to-current map defined by $\Lambda \phi \coloneqq \frac{\partial u}{\partial \normal}$. Is it possible to determine the conductivity $\gamma$ from the voltage-to-current map $\Lambda$?

One avenue of active research seeks to solve a more general problem called the Calder\'on problem for Riemannian manifolds \cite{joshi_inverse_2005,belishev_dirichlet_2008,krupchyk_inverse_2011,sharafutdinov_complete_2013,shonkwiler_poincare_2013}. In this setting, one attempts to reconstruct an $n$-dimensional smooth manifold $M$ and metric $g$ from the Dirichlet-to-Neumman (DN) map $\Lambda$ defined on the boundary traces of harmonic differential forms, i.e., forms in the kernel of the Laplace--Beltrami operator $\Delta$. Given the Dirichlet boundary condition $\phi$, we have a unique harmonic $u$. The DN map is defined by $\Lambda \phi \coloneqq \iota^*(\star du)$ which is valid for any $k$-form $\phi$. Fixing $\phi$ to be a 0-form then yields the classical problem discussed previously and in dimension 3, this is equivalent to the EIT problem. From complete knowledge of pairs of Dirichlet and Neumann data $(\phi, \iota^* (\star du))$ for forms of all grade, the geometric inverse problem is to determine the pair $(M,g)$ up to isometry. If you would like, please see Uhlmann's article \cite{uhlmann_inverse_2014} for an excellent explanation of the relationship between the EIT and Calder\'on problem.

Solutions to the Calder\'on problem exist in a handful of special cases, but the $C^\infty$-smooth problem remains open in dimensions greater than two. Belishev \cite{belishev_calderon_2003} and Lassas--Uhlmann \cite{lassas_determining_2001} both show that for surfaces $S$ with one boundary component, the classical DN map determines $S$ up to conformal class. A better result cannot be achieved since the Laplacian is conformally invariant in dimension 2. Belishev, Badanin, and Korikov recently achieved a similar result for surfaces with multiple boundary components \cite{badanin_electric_2021}. It is also known that the DN map for forms recovers partial topological information on arbitrary $M$ such as the Betti numbers \cite{belishev_dirichlet_2008}, but it is not known whether the DN operator can recover $M$ up to homeomorphism. The extension of the DN operator to the complete DN operator in Sharafutdinov and Shonkwiler's paper \cite{sharafutdinov_complete_2013} is able to recover the absolute and relative cohomologies as well.

We will mostly concern ourselves with the technique used by Belishev in his 2003 paper \cite{belishev_calderon_2003} called the \emph{Boundary Control (BC) method}. This technique utilizes the classical Gelfand representation for commutative Banach algebras. That is, for a surface $S$, the spectrum (maximal ideal space) of the commutative Banach algebra of holomorphic functions is homeomorphic to $S$. Alongside this result, Belishev then used the complex structure of the algebra to determine the metric up to conformal equivalence. Furthermore, it is the Gelfand transform that allows for all of this information to be obtained from the boundary. This led to a solution to the Calder\'on problem in dimension 2. For more on the boundary control method, see \cite{belishev_boundary_2017}. 

To solve the problem in dimension three using the BC method, it is natural to consider replacing complex functions with quaternion-valued functions. Belishev and Vakulenko wrote a series of papers on this topic \cite{belishev_algebras_2017,belishev_algebras_2020,belishev_algebraic_2019}. In those papers, the authors work with the space of harmonic quaternion fields using the language of differential forms. Their goal was to complete one portion of the BC method by realizing a Gelfand theory for quaternion fields. The first hurdle in defining a Gelfand spectrum on the space of harmonic quaternion fields is that the space fails to be an algebra. Moreover, it is not commutative. However, the space does contain commutative subalgebras and by carefully defining a meaningful notion of a spectrum that is multiplicative over these subalgebras, they find that for convex regions in $\R^3$, the spectrum is homeomorphic to the ball. Is a similar result true for a more general class of manifolds? 

In this paper, we seek to faithfully capture the necessary Gelfand theory in dimension 2 and 3 as well as generalize it to arbitrary dimension by replacing the exterior algebra of differential forms with Clifford algebras of multivector fields. Note that we do not lose any information with this substitution since the exterior algebra naturally embeds into in any Clifford algebra. Our focus will be on nondegenerate (Euclidean) Clifford algebras $\G$ which we call \emph{(Euclidean) geometric algebras} since they  capture both complex and quaternion algebras as spinor subalgebras $\G^+\subset \G$ when defined on $\R^2$ and $\R^3$ resectively. Many of the nice properties of complex and quaternion valued functions can be generalized to higher dimensions by replacing them with spinor fields $f_+ \in \smoothfields{M}{\G^+}$ which are sections of the bundle of the spinor subalgebras.

A complex function on a surface can be written as a sum of even grade forms, specifically, a 0-form $\omega$ and 2-form $\eta$. This function $\omega + \eta$ is holomorphic if and only if the generalized Cauchy--Riemann equation $d \omega = -\delta \eta$ where $d$ is the exterior derivative and $\delta$ is the codifferential is satisfied. Such $\omega$ and $\eta$ are called a conjugate pair. The same is true for a harmonic quaternion fields. Instead of looking at pairs of conjugate forms, we consider spinor fields in the kernel of the Hodge--Dirac operator $\grad$. This operator is fundamental in Clifford analysis and it is equivalent to operator $d+\delta$ on forms. Spinor fields in the kernel of $\grad$ are not restricted to being only pairwise conjugate and the associated Cauchy--Riemann equations are a bit more general in dimensions higher than three. We refer to such fields in the kernel of the Hodge--Dirac operator as \emph{monogenic}. 

As a side note, if we consider fields of only a single grade, if they are monogenic then they correspond directly to harmonic fields. This lands us in the realm of Hodge theory where it is a central result that for each grade $k$, the space of harmonic fields is finite dimensional. Fundamentally, by considering fields consisting of sums of even grades (i.e., a spinor), we find the kernel to be far more rich in content. For example, in $\R^2$, monogenic spinor fields are equivalent to holomorphic functions and in $\R^3$, they are equivalent to harmonic quaternion fields and both of these spaces are infinite in dimension. The real benefit of using Clifford algebras is that it allows us to be dimension-agnostic like we can with differential forms, but recover many of the staple complex analysis theorems in any given dimension as well.

Within this framework, we prove a Gelfand representation for $n$-dimensional compact regions of $\R^n$ and define the corresponding Gelfand transform. Here, the spectrum consists of a certain type of functionals which we call \emph{$\G$-currents} that act on the space of monogenic spinor fields. These functionals are valued in the base geometric algebra and respect the structure of certain subalgebras nested inside the space of monogenic spinor fields. Furthermore, we report a Stone--Weierstrass theorem showing that the algebra generated by the closure of the monogenic fields is dense in the space of continuous fields. Both of the above results answer questions posed by Belishev and Vakulenko in \cite{belishev_algebras_2020}. 

The construction we use for the spectrum bootstraps from our knowledge of surfaces: given a surface $S$, the monogenic spinor fields on the surface $\monogenics^+(S)$ are a commutative Banach algebra isomorphic to the algebra of holomorphic functions. When the dimension of the manifold $M$ exceeds two, the space $\monogenics^+(M)$ is not an algebra since products of monogenic fields need not be monogenic. At a local scale, a monogenic spinor $f_+\in \monogenics^+(M)$ is, in essence, built out of a monogenic fields propagated off of surfaces embedded in $M$. In fact, we can think of these fields as a local set of variables and this lets us write power series for $f_+$ locally in these variables. We find that these variables are direct analogs to the variable $z$ in complex analysis. 

The key challenge we face in this paper is to define a notion of a spectrum for the space $\monogenics^+(M)$. In the case of a surface $S$, characters $\delta$ in the spectrum $\characters(S)$ are the continuous algebra morphisms $\characters(S) \to \C$. Since $\monogenics^+(M)$ is not an algebra but is a $\G^+$-module, we define characters $\characters(M)$ to be continuous module morphisms $\monogenics^+(M) \to \G^+$. Furthermore, we require that characters are also algebra morphisms on subalgebras $\algebra{\bivector}$ (which are akin to $\monogenics^+(S)$ for $S$ a subsurface in $M$) to a an algebra $\planespinors$ (which is isomorphic to $\C$). Applying such characters to the power series representation for $f_+$ in terms of the $z$ variables yields our main theorem.
\begin{theorem}
\label{thm:gelfand_intro}
Let $M$ be a compact region in $\R^n$. For any $\delta \in \characters(M)$, there is a point $\blade{x}_\delta \in M$ such that $\delta(f) = f(\blade{x}_\delta)$ for any $f_+\in \monogenics^+(M)$ a monogenic field. Given the weak-$\ast$ topology on the space of $\G$-currents, the map
\[
\Gamma \colon \characters(M) \to M, \quad \delta \mapsto \blade{x}_\delta
\]
is a homeomorphism. The Gelfand transform $\monogenics^+(M) \to \contfields{\characters(M)}{\G^+}$ given by $\widehat{f_+}(\delta) = \delta[f_+]$ is an isometric isomorphism onto its image so that $\monogenics^+(M)\cong \widehat{\monogenics^+(M)}$.
\end{theorem}
Furthermore, we find that the closure $\overline{\monogenics^+(M)}$ separates points which follows from the fact from Calderbank's thesis \cite{calderbank_geometrical_1995}. Namely, if we know $f_+ \in \monogenics^+(M)$ locally, we can uniformly extend this to a unique field on all of $M$. Given this and a theorem from Laville and Ramadanoff's paper \cite{laville_stone-weierstrass_1996}, we prove the following Stone--Weierstass result.
\begin{theorem}
\label{thm:stone_weierstrass_intro}
Let $\vee \overline{\monogenics^+(M)}$ represent the minimal algebra generated by $\overline{\monogenics^+(M)}$. Then $\vee \overline{\monogenics^+(M)}$ is dense in $\contfields{M}{\G^+}$.
\end{theorem}
Lastly, I hope that others find that this work warrants further investigation on the utility of $C^\ast$- and Banach-algebras of Clifford algebra-valued functions and to this end I will finish the paper with a discussion of other questions addressed by Belishev and Vakulenko and their relation to this work and the Calder\'on problem. It is certainly a worthy endeavor to consider what extent these theorems can be used when the information is extracted solely from Dirichlet-to-Neumann operators.

\subsection*{Acknlowedgements:}
Without the unending guidance and input from my advisor Dr. Clayton Shonkwiler, none of this work would be possible. I deeply appreciate your help and your willingness to follow me on my own mathematical journey over these past few years. Thank you.

\section{Preliminaries}
\label{sec:preliminaries}

We address only the necessary preliminaries here, though more detail can be found in various sources, for instance from Doran and Lasenby \cite{doran_geometric_2003}, and Hestenes and Sobczyk \cite{hestenes_clifford_1984}. Our preliminaries will include the basics of Clifford algebras with a large motivating example, foundations of the Clifford analysis of the Hodge--Dirac operator, and the notion of algebraic currents.

\subsection{Clifford Algebras}
\label{subsec:clifford_algebras}

To construct any Clifford algebra, take an $n$-dimensional vector space $V$ over a field $\mathbb{F}$ with quadratic form $Q$. By the polarization identity, given a $Q$ there is a unique corresponding symmetric bilinear form $g$. For more information on vector spaces with quadratic or bilinear forms, please see \cite{roman_metric_2008}.

The tensor algebra is given by $\bigoplus_{n \in \mathbb{N}} V^{\otimes^n}$ and we form the Clifford algebra $\clifford(V,Q)$ by the quotient
\begin{equation}
\label{eq:ideal_quotient}
\clifford(V,Q) \coloneqq \bigoplus_{n \in \mathbb{N}} V^{\otimes^n} ~ / ~ \langle \blade{v} \otimes \blade{v} - Q(\blade{v}) \rangle
\end{equation}
with the induced addition and multiplication from the tensor algebra. For sake of clarity, we will think of $\clifford(V,Q)=\clifford(V,g)$ as needed, though most define Clifford algebras only using quadratic forms. Elements of $\clifford(V,Q)$ are referred to as \emph{multivectors}. 

When we take the totally singular form $Q=0$, the corresponding Clifford algebra is the exterior algebra $\bigwedge(V) = \clifford(V,0)$, and otherwise $\bigwedge(V)\subseteq C\ell(V,Q)$ as a subalgebra. Note that $\clifford(V,Q)$ is a $\mathbb{F}$-vector space of dimension $2^{n}$. For the remainder of this paper, we take $\mathbb{F}=\R$.

Given linearly independent vectors $\blade{v}_1,\dots,\blade{v}_k$ and the exterior product $\wedge$, an element of the form
\begin{equation}
    \blade{A_k} = \blade{v}_1 \wedge \cdots \wedge \blade{v}_k
\end{equation}
is a \emph{$k$-blade}, which are the simplest multivectors. Using the $\wedge$ inherently removes lower grade elements and keeps only the highest grade element (grade-$k$) of the product $A=\blade{v}_1\blade{v}_2\cdots \blade{v}_k$ since, at the very least, the product of two vectors yields \cref{eq:product_of_vectors}. This multivector $A$, often called a \emph{versor} since it is a product of vectors, is a sum of different graded elements called \emph{$k$-vectors}. A $k$-vector is a linear combination of $k$-blades and we denote this subspace of grade-$k$ elements by $\clifford^k(V,Q)$. Therefore, we have the direct sum decomposition
\begin{equation}
    \label{eq:grade_decomposition}
    \clifford(V,Q)= \bigoplus_{k=0}^n \clifford^k(V,Q).
\end{equation}
Some may refer to $k$-blades as \emph{simple} or \emph{decomposable} $k$-vectors as they correspond to rank-1 tensors in the tensor algebra. Moreover, they are also representative of subspaces which we discuss later. We write $\clifford^{k\oplus \ell}(V,Q)$ to represent the direct sum space of $k$- and $\ell$-vectors. A basis $\blade{e}_i$ of $V$ induces a $k$-blade basis by taking $\mathcal{I}=\{i_1,\dots,i_k\}$ to be a list of increasing indices $i_1 < \cdots < i_k$ and putting
\begin{equation}
\label{eq:basis_blades}
\blade{E}_\mathcal{I} \coloneqq \blade{e}_{i_1}\wedge \cdots \wedge \blade{e}_{i_k}.
\end{equation}

Some graded elements have special names. We say grade-0 objects are \emph{scalars}, grade-1 are \emph{vectors}, grade-2 are \emph{bivectors}, and grade-$n$ objects are \emph{pseudoscalars}. Typically, objects of grade-$(n-k)$ receive the prefix ``pseudo'', for example, we have pseudoscalars of grade-$(n-0)$ and \emph{pseudovectors} of grade-$(n-1)$. Given the direct sum decomposition in \cref{eq:grade_decomposition}, a multivector $A\in \clifford(V,Q)$ is given by
\begin{equation}
    \label{eq:grade_decomp_of_multivector}
A = \sum_{k=0}^n A_k
\end{equation}
where $A_k \in \clifford(V,Q)^k$.

Multivectors also split into even and odd grades. We will work with the even-graded elements called \emph{spinors} as they form their own subalgebra $\clifford^+(V,Q)$. Due to this, some refer to $\clifford(V,Q)$ as a \emph{superalgebra} since there is this $\mathbb{Z}/2\mathbb{Z}$-splitting. Spinors may also be defined slightly more generally (see \cite{calderbank_geometrical_1995}), but this notion suffices for this paper.

One purpose of using a Clifford algebra $\clifford(V,Q)$ is to extend the exterior algebra $\bigwedge(V)$ to include a useful interior multiplication. Given $\blade{v},\blade{w}\in \clifford(V,Q)^1$, their product splits as
\begin{equation}
\label{eq:product_of_vectors}
    \blade{v}\blade{w}= \underbrace{\blade{v}\cdot \blade{w}}_{\textrm{grade-0}} +\underbrace{\blade{v}\wedge \blade{w}}_{\textrm{grade-2}},
\end{equation}
where $\cdot$ is the \emph{interior product}. However, for general Clifford algebras, we may have \emph{degenerate} vectors $\blade{v}$ such that for any other vector $\blade{w}$, $\blade{v}\cdot \blade{w}=0$. We will want to rid of this case.

To remove degenerate vectors from the algebra, we can force the quadratic form $Q$ to be completely nonsingular. We refer to such a Clifford algebra with nonsingular $Q$ as a \emph{geometric algebra} and we write $\G = \clifford(V,Q)$ to denote such an algebra. To see more reasoning of calling such algebras ``geometric'', we again refer the reader to the chapter \cite{roman_metric_2008}. Using the vector space basis, we can determine the coefficients of the bilinear form by
\begin{equation}
g_{ij}=g(\blade{e}_i,\blade{e}_j)=\blade{e}_i \cdot \blade{e}_j.
\end{equation}
Let us quickly remark that if we were considering a Clifford algebra that was not a geometric algebra, then the matrix $g_{ij}$ of the bilinear form would be singular. We will carry on the rest of this paper working solely with geometric algebras for this reason, so the reader can assume that we take $g$ whose matrix representations $g_{ij}$ are full rank and symmetric.

Given a $k$-vector $A_k$ and an $\ell$-vector $B_\ell$, the product is
\begin{equation}
\label{eq:general_clifford_multiplication}
A_k B_\ell = \proj{|k-\ell|}{A_k B_\ell} + \proj{|k-\ell|+2}{A_k B_\ell} + \cdots + \proj{k+\ell}{A_k B_\ell},
\end{equation}
where the brackets $\proj{k}{-} \colon \G \to \G^k$ denote projection into the grade-$k$ subspace by
\begin{equation}
\proj{k}{A}=A_k
\end{equation}
when $A$ is given by \cref{eq:grade_decomp_of_multivector}. We define the \emph{(left) contraction} $A_k \contract B_\ell \coloneqq \proj{\ell-k}{A_k B_\ell}$. In general, the lowest grade term of $A_kB_\ell$ is the interior product $A_k \cdot B_\ell = \proj{|k-\ell|}{A_kB_\ell}$ and the exterior product $\wedge$ is the highest grade term of the product so that $A_k \wedge B_\ell = \proj{k+\ell}{A_k B_\ell}$. For a vector $\blade{v}$ we have $\blade{v}\cdot A = \blade{v}\contract A$ so many equations can be written with either $\cdot$ or $\contract$. Most will be written with $\contract$ as it is algebraically and geometrically more convenient. For notational simplicity, we also remove the subscript when projecting into the scalar subspace, $\proj{}{-}=\proj{0}{-}$, but this should not be confused with the notation for the ideal generated by a relation used only in \cref{eq:ideal_quotient}.

The \emph{reciprocal basis vectors $\blade{e}^i$} are those that satisfy $\blade{e}^i \cdot \blade{e}_j = \delta^i_j$. Reciprocal basis elements allow us to use the Riesz representation in order to avoid extraneous use of dual space elements since we are able to capture this functionality through the interior product. For sake of clarity, $\blade{e}^i \cdot \blade{e}^j = g^{ij}$ is the matrix inverse to $g_{ij}$ and $\blade{e}^i = g^{ij} \blade{e}_j$ are just the ``raised up'' indices. For a basis blade $\blade{E}_\mathcal{I}$, the reciprocal blade is $\blade{E}^\mathcal{I}$ and it satisfies the equation $\blade{E}^\mathcal{I} \cdot \blade{E}_\mathcal{J}= \delta^\mathcal{I}_\mathcal{J}$ where $\delta^\mathcal{I}_\mathcal{J}=1$ only when the sets of indices $\mathcal{I}$ and $\mathcal{J}$ are identical.

Geometric algebras have a bilinear product $\G \times \G \to \R$ called the \emph{multivector inner product} which is given by
\begin{equation}
(A,B) \coloneqq \proj{}{A^\dagger B}.
\end{equation}
This equation is given in terms of the \emph{reverse operator} $\dagger$ which for $\lambda \in \R$ satisfies
\begin{equation}
(A+B)^\dagger=A^\dagger + B^\dagger, \quad (\lambda A)^\dagger = \lambda^\dagger A^\dagger = \lambda A^\dagger, \quad A^{\dagger \dagger}=A, \quad (AB)^\dagger = B^\dagger A^\dagger,
\end{equation}
and on a versor we have
\begin{equation}
    (\blade{v}_1\blade{v}_2 \cdots \blade{v}_k)^\dagger \coloneqq \blade{v}_k \cdots \blade{v}_2 \blade{v}_1.
\end{equation}
Note that $\dagger$ acts as the adjoint in the product $(-,-)$ which follows from the cyclic property of the scalar grade projection \cite[eq. (138)]{chisolm_geometric_2012}. To see this, we take another multivector $C$ and note
\begin{align}
(CA,B) = \proj{}{(CA)^\dagger B} = \proj{}{ A^\dagger C^\dagger B } = (A,C^\dagger B),
\end{align}
We define a semi-norm $|-|^2\coloneqq (-,-)$ called the \emph{multivector norm}. If $|A|=\pm 1$ we say that $A$ is \emph{unit}. It is worth saying that for a multivector field written in terms of basis blades $f=\sum_{\mathcal{I}} f_\mathcal{I} \blade{E}_\mathcal{I}$ that
\begin{equation}
\label{eq:inner_product_with_basis}
f_\mathcal{I} = (f, \blade{E}^\mathcal{I}),
\end{equation}
so long as the quadratic form is definite (e.g., the quadratic form is the Euclidean norm). We also have that
\begin{equation}
\blade{E}^\mathcal{I} = (\blade{e}^{i_1} \wedge \cdots \wedge \blade{e}^{i_k})^\dagger.
\end{equation}

There exists a vector basis for $V$ where $p$ vectors square to $-1$ and $q$ vectors square to $+1$ and $p+q=n$. The corresponding geometric algebra is often written as $\G_{p,q}$. We will focus most on the the case where $Q$ is positive definite and to distinguish this, we can write $\G_n$ if it is necessary. For $\G_n$, the multivector inner product and multivector norm are both positive definite. One can see that the multivector inner product treats the space $\G_n$ as a $2^n$-dimensional inner product space with a basis given by the blades $\blade{E}_\mathcal{I}$. \Cref{eq:inner_product_with_basis} just specifies that we have chosen a set of blades orthonormal with respect to the multivector inner product.

If the basis $\blade{e}_i$ is orthonormal in $V$, then the set of basis blades $\blade{E}_\mathcal{I}$ are orthonormal versors in $\G_n$ since
\begin{equation}
\blade{E}_\mathcal{I} = \blade{e}_{i_1}\wedge \cdots \wedge \blade{e}_{i_k} = \blade{e}_{i_1} \blade{e}_{i_2} \cdots \blade{e}_{i_k}.
\end{equation}
Their products become much clearer to compute. We have
\begin{equation}
\blade{E}_\mathcal{I} \blade{E}_\mathcal{J} =  \pm \blade{E}_{\mathcal{I}\triangle \mathcal{J}},
\end{equation}
where $\triangle$ is the symmetric difference of the sets $\mathcal{I}$ and $\mathcal{J}$ and the $\pm$ is used solely due to the fact that vectors $\blade{e}_i$ comprising the versors $\blade{E}_\mathcal{I}$ may need to be swapped and
\begin{equation}
-\blade{E}_\mathcal{I} = \blade{e}_{i_1} \blade{e}_{i_2} \cdots \blade{e}_{i_{j+1}} \blade{e}_{i_j}\cdots \blade{e}_{i_k}.
\end{equation}
For a concrete example, take $\blade{E}_{123}=\blade{e}_1\blade{e}_2\blade{e}_3$ and $\blade{E}_{124}$ both in $\G_n$, then
\begin{equation}
\blade{E}_{123}\blade{E}_{124}= \blade{e}_1\blade{e}_2\blade{e}_3\blade{e}_1\blade{e}_2\blade{e}_4=\blade{e}_1\blade{e}_2\blade{e}_1\blade{e}_2\blade{e}_3\blade{e}_4 = -\blade{e}_1\blade{e}_2\blade{e}_2\blade{e}_1\blade{e}_3\blade{e}_4 = -\blade{e}_1\blade{e}_1\blade{e}_3\blade{e}_4 = -\blade{E}_{34}.
\end{equation}
Using an orthonormal vector basis shows how nicely versors act algebraically. Multiplication is just reduction of words in the characters $\blade{e}_i$ subject to the relations $\blade{e}_i^2=1$ and  $\blade{e}_i\blade{e}_j =- \blade{e}_j\blade{e}_i$ when $i\neq j$.

\begin{remark}
Though we will not cover the content here, it is worth mentioning that versors in a geometric algebra $\G_{p,q}$ form a group under multiplication called the \emph{Clifford group} and the unit versors define the \emph{spin group} $\mathrm{Spin}(p,q)$. The algebra of bivectors in $\G_{p,q}$ with the commutator $[-,-]$ (often written as $\times$ as well) is the Lie algebra $\mathfrak{spin}(p,q)$.
\end{remark}

Geometric algebras also have a unique isomorphism $\perp \colon \G^k \to \G^{n-k}$ and this $\perp$ is equivalent to the Hodge star $\star$ in $\bigwedge(V)$. To construct this isomorphisms, we first take an arbitrary basis for $V$ and build the volume element
\begin{equation}
    \mu \blade{I} \coloneqq \blade{e}_1 \wedge \cdots \wedge \blade{e}_n,
\end{equation}
where the unit $n$-blade $\blade{I}$ is the \emph{unit pseudoscalar} and the scalar $\mu$ represents the volume scaling. Note that $\blade{I}$ represents the subspace $V$ and for geometric algebras it defines the \emph{dual}
\begin{equation}
\label{eq:dual}
A^\perp \coloneqq A\pseudoscalar^{-1}.
\end{equation}
For example, if $\blade{v}$ is a vector then we have $\blade{v}^\perp$ is a pseudovector representing a scaled copy of the hyperplane perpendicular to $\blade{v}$.

Moreover, the dual allows for exchanging products
\begin{equation}
(A\contract B)^\perp = A\wedge B^\perp
\end{equation}
and the exterior product
\begin{equation}
(A\wedge B)^\perp = A\contract B^\perp
\end{equation}
and helps elucidate the geometrical meaning of $\contract$ and $\wedge$ (see \cite{chisolm_geometric_2012}). Finally, it is worth saying that in $\G_n$ we have $\pseudoscalar^{-1}=\pseudoscalar^\dagger$.

Given a $k$-dimensional subspace $U\subset V$ in a space with nonsingular $Q$, we can put $V=U \oplus U^\perp$ where $U^\perp$ is an $n-k$-dimensional subspace. In much the same way, given a unit $k$-blade $\blade{U_k}$ we can find a decomposition of $\pseudoscalar$ by $\pseudoscalar = \blade{U_k}\wedge \blade{U_k}^\perp$. Actually, multiplication in $\G$ allows for projection onto subspaces using this identification.
\begin{definition}
Given an multivector $B$ and unit $k$-blade $\blade{U_k}$, the \emph{projection} onto the subspace $\blade{U_k}$ is
\begin{equation}
\label{eq:projection}
\projection_{\blade{U_k}}(B) \coloneqq (B\contract \blade{U_k}) \blade{U_k}^{-1}.
\end{equation}
\end{definition}
The projection preserves grades $\projection_{\blade{U_k}}(B_\ell)\in \G^\ell$. A specific application of the projection is to view spinors along planes or, eventually, surfaces. This will be our key methodology to bootstrap from complex analysis.
\begin{definition}
Let $\bivector$ be a unit 2-blade, then the space of \emph{plane spinors} are the elements
\begin{equation}
 \planespinors \coloneqq \R \oplus \Span(\bivector).
\end{equation}
\end{definition}
As we will see in the following example, $\G_2^+\cong \C$ and $\G_3^+\cong \mathbb{H}$. Furthermore, for any $\bivector \in \G_n$ we also have $\planespinors \cong \C$. Hence, just like the Grassmannian of planes in $\R^3$, $\mathrm{Gr}(2,3)$, parameterizes copies of $\C$ in $\mathbb{H}$ by choice of imaginary unit, plane spinors $\planespinors$ as a subalgebra of $\G^+_n$ are parameterized by the Grassmannian $\mathrm{Gr}(2,n)$. If you find the following example lacking, then Doran \& Lasenby's text \cite{doran_geometric_2003} is very insightful and filled with great intuition.

\begin{example}
Rather than a sequence of multiple examples, it will prove to be far more illuminating to construct one large example for which most of the preliminaries to this point can be used in a meaningful way. As such, we shall not rule out the utility that other researchers may gain out of using geometric algebras with pseudo inner products even though this paper is predominantly concerned with the positive definite case. The classical example is the \emph{spacetime algebra} defined by taking $V=\R^4$ with a vector basis $\blade{e}_0,\blade{e}_1,\blade{e}_2,\blade{e}_3$ satisfying
\begin{subequations}
\begin{align}
\blade{e}_0 \cdot \blade{e}_0 &= -1\\
\blade{e}_0 \cdot \blade{e}_i &= 0  &i=1,2,3\\
\blade{e}_i \cdot \blade{e}_j &= \delta_{ij}, &i,j=1,2,3.
\end{align}
\end{subequations}
We refer to $\blade{e}_0$ as \emph{temporal} since its square is negative and $\blade{e}_i$ for $i=1,2,3$ are \emph{spatial} since their squares are positive. For this basis, the matrix for this inner product in this basis assumes the form $\eta =\operatorname{diag}(-1,~ +1,~ +1,~+1)$ (often called the \emph{Minkowski metric}). The associated quadratic form $Q$ can be found from $\eta$ by polarization. For a spacetime vector $\blade{v} = v_0 \blade{e}_0 +v_1 \blade{e}_1 + v_2 \blade{e}_2 + v_3 \blade{e}_3$,
\begin{equation}
\label{eq:spacetime_inner_product}
|\blade{v}|^2 = (\blade{v},\blade{v}) = \blade{v} \cdot \blade{v} = -v_0^2 + \sum_{i=1}^3 v_i^2.
\end{equation}
It is clear that the norm is definite when all vectors are spatial, but in the case of spacetime there are null vectors $\blade{c}$ such that $|\blade{c}|=0$. For example, $\blade{c}=\blade{e}_0 + \blade{e}_1$. The collection of null vectors define the light cone in Minkowski space. Also, it is important to distinguish these null vectors from degenerate vectors. Though $\blade{c}$ is null, it is not true that for any $\blade{c}$ and all other vector $\blade{v}$ that $\blade{c} \cdot \blade{v}=0$. This is only true for other $\blade{v}$ on the light cone and the light cone is not a subspace.

As the notation above suggests, the geometric algebra of Euclidean space $\R^3$, $\G_3$, should naturally appear inside of the spacetime algebra. The spatial \emph{trivector} $\blade{e}_1 \blade{e}_2 \blade{e}_3$ is unit
\begin{equation}
|\blade{e}_1 \blade{e}_2 \blade{e}_3| = \sqrt{\proj{}{(\blade{e}_1 \blade{e}_2 \blade{e}_3)^\dagger \blade{e}_1 \blade{e}_2 \blade{e}_3}} = \sqrt{\proj{}{\blade{e}_3 \blade{e}_2 \blade{e}_1 \blade{e}_1 \blade{e}_2 \blade{e}_3}}=1
\end{equation}
and represents the spatial subspace $\mathrm{Span}(\blade{e}_1, \blade{e}_2, \blade{e}_3) \subset \R^4$. With slight abuse of notation, the projection of $\G_{1,3}$ onto this subspace yields
\begin{equation}
\projection_{\blade{e}_1 \blade{e}_2  \blade{e}_3}(\G_{1,3}) = \G_3.
\end{equation}
In $\G_3$, we can specify an arbitrary multivector $A$ by
\begin{equation}
A= a_0 + a_1 \blade{e}_1 + a_2 \blade{e}_2 + a_3 \blade{e}_3 + a_{12} \blade{e}_1\blade{e}_2 + a_{13} \blade{e}_1\blade{e}_3 + a_{23} \blade{e}_2\blade{e}_3 + a_{123} \blade{e}_1 \blade{e}_2 \blade{e}_3.
\end{equation}
It will be pertinent later to define $\bivector_{ij}\coloneqq \blade{e}_i\blade{e}_j$ for $i\neq j$. Using this substitution, the grade projections read
\begin{subequations}
\begin{align}
\proj{}{A}&=a_0\\
\proj{1}{A}&=a_1 \blade{e}_1 + a_2 \blade{e}_2 + a_3 \blade{e}_3\\
\proj{2}{A}&=a_{12} \blade{B}_{12} + a_{13} \blade{B}_{13} + a_{23} \blade{B}_{23}\\
\proj{3}{A}&= a_{123} \blade{e}_1  \blade{e}_2 \blade{e}_3.
\end{align}
\end{subequations}
Hence, we can write a spinor as
\begin{equation}
A_+ = a_0 + a_{12} \blade{B}_{12} + a_{13} \blade{B}_{13} + a_{23}\blade{B}_{23}.
\end{equation}
Note as well that the spatial unit 2-blades always satisfy
\begin{align}
\blade{B}_{23}^2 = \blade{B}_{13}^2 = \blade{B}_{12}^2 = -1
\end{align}
and we find that
\begin{equation}
\blade{B}_{23}\blade{B}_{13}\blade{B}_{12} = -1.
\end{equation}
Hence, the even subalgebra $\G_3^+$ isomorphic to the quaternion algebra $\quat$ by
\begin{equation}
\mathbf{i} \leftrightarrow \blade{B}_{23}, \quad \mathbf{j} \leftrightarrow \blade{B}_{13}, \quad \mathbf{k} \leftrightarrow \blade{B}_{12}
\end{equation}
Given a quaternion, there is an equivalent spinor $A_+$; the imaginary part of the quaternion corresponds to the grade two part of the spinor $\proj{2}{A_+}$.

We can project down one dimension further by $\projection_{\bivector_{12}} (\G_3) = \G_2$ and
we can verify quickly that
\begin{subequations}
\begin{align}
    \projection_{\bivector_{12}}(A) =  a_0 + a_1 \blade{e}_1 + a_2 \blade{e}_2 + a_{12} \bivector_{12}.
\end{align}
\end{subequations}
Given that $\blade{B}_{12}^2=-1$ we can put $z \coloneqq x + y \blade{B}_{12} \in \G_2^+$ for $x,y\in \R$ which is exactly a representation of the complex number $\zeta = x+ \mathbf{i}y$ in $\C$ and $\mathbf{i}$ here can be thought of as the unit pseudoscalar in the plane. Again, the imaginary part is $\proj{2}{z}$.

But, the above work is not special to the starting point of $\G_{1,3}$ or $\G_3$. In fact, if we take $\G_n$ for $n\geq 2$, then there are natural copies of $\C$ contained inside of $\G_n$. In particular, we have the isomorphism
\begin{equation}
\label{eq:c_isomorphisms}
    \C \cong \{x + y \blade{B} ~\vert~ x,y \in \R,~ \blade{B} \in \Grassmannian{2}{n}. \},
\end{equation}
which shows that complex numbers arise as plane spinors via the representation $\zeta = x + y\blade{B}$. Given the standard basis $\blade{e}_1,\dots,\blade{e}_n$ we have the ${ n \choose 2}$ unit bivectors $\blade{B}_{ij}$ for $j=1,\dots,n$ and $i<j$. The plane spinors $\mathbb{A}_{\bivector_{ij}}$ are each isomorphic to $\C$.
\end{example}

\subsection{Clifford Analysis}
\label{subsec:clifford_analysis}

Given a smooth semi-Riemannian manifold $M$ with metric tensor field $g$ and boundary $\partial M$, each tangent space $T_xM$ can be made into a geometric algebra $\G_x M \coloneqq C\ell(T_xM,g_x)$ and we call $\G_xM$ the \emph{geometric tangent space}. The geometric tangent spaces are glued together to form the geometric algebra bundle $\G M\coloneqq \bigsqcup_{x\in M}\G_x M$. We will call the $C^\infty$ sections of this bundle $\smoothfields{M}{\G}$ the \emph{smooth multivector fields} and the continuous sections $\contfields{M}{\G}$ the \emph{continuous multivector fields}. Let $\nabla$ be the Levi--Civita connection on $M$. Then for a vector field $\blade{v}$ we have the covariant derivative $\nabla_{\blade{v}}$ which is extended to multivector fields, e.g., by \cite{schindler_geometric_2020}. Given local coordinates $x^i$ on $M$ we have the induced (gradient) basis $\blade{e}_i(x)\in \G_xM$. Suppressing the pointwise notation, $\blade{e}_i\cdot \blade{e}_j = g_{ij}$. Also in the tangent space are the reciprocal vectors $\blade{e}^i$ which are Riesz representatives corresponding to the dual basis $dx^i$ since $dx^i(\blade{e}_j) = \blade{e}^i \cdot \blade{e}_j$. We will not get rid of the dual basis entirely since we still find use for the elements as coordinate measures in integration. Basis blades $\blade{E}_\mathcal{I}$ and their reciprocals $\blade{E}^\mathcal{I}$ carry over to each geometric tangent space as well.

The \emph{Hodge--Dirac operator} (or \emph{gradient}) $\grad$ in these coordinates is
\begin{equation}
    \grad \coloneqq \sum_{i=1}^n \blade{e}^i \nabla_{\blade{e}_i}.
\end{equation}
This derivative acts algebraically as a vector and
\begin{equation}
\label{eq:grad_of_planespinor}
\grad f = \grad \contract f + \grad \wedge f.
\end{equation}
In particular, $\grad \wedge$ is equivalent to the exterior derivative $d$ on differential forms $\Omega(M)$, $\grad \contract$ is equivalent to the codifferential $\delta$, and $\grad^2 = \Delta$ is the Laplace--Beltrami operator. The kernel of $d+\delta$ on $\Omega^k(M)$ is coupled to the (co)homology of $M$ and so $\grad$ retains this property as well. This relationship of the analysis of $d+\delta$ (and equivalently $\grad$) to the absolute and relative (co)homology of $M$ is formalized in Hodge theory and is useful for proving existence and uniqueness for boundary value problems \cite{schwarz_hodge_1995}. However, in Clifford analysis, it is commonplace to consider mixed grade multivectors also in the kernel of $\grad$. This space is far bigger since we find that grades can ``mix'' together. For example, take $M$ to be the unit disk $\mathbb{D} \subset \R^2$ and $f_+ = f_0 + f_2 \bivector$ with $f_0,f_2 \in \smoothfields{\mathbb{D}}{\R}$ where $\bivector$ is the unit 2-blade field. Then we have
\begin{equation}
\grad f_+ = \grad \wedge f_0 + \grad \contract ( f_2 \bivector).
\end{equation}
If we considered only singly graded elements such as the scalar fields $f_0$ or bivector fields $f_2 \bivector$ on their own, then the only elements in the kernel of $\grad$ are constant fields. On the other hand, when we combine them together into a spinor field $f_+ \in \smoothfields{\mathbb{D}}{\G^+}$ then this $f_+$ now be any holomorphic function such as $z=x^1+x^2\bivector$ or even any power series in $z$ which we will revisit later. Since this space is far bigger, it can be used to extract more topological data about $M$, namely the homeomorphism type, as we see in \cref{thm:gelfand}.

\begin{definition}
Let $f\in \smoothfields{M}{\G}$, then we say that $f$ is \emph{monogenic} if $\grad f = 0$. We denote the space of monogenic fields by $\monogenics(M)$.
\end{definition}

Monogenic fields are the emphasis of Clifford analysis and many of the theorems of holomorphic functions in complex analysis apply to these fields. Once again, on the unit disk $\mathbb{D}$ take $f_+ = f_0 + f_2 \bivector$ then if $\grad f_+ =0$ we can use \cref{eq:grad_of_planespinor} to derive the Cauchy--Riemann equations
\begin{align}
\label{eq:cauchy_riemann_equations}
\frac{\partial f_0}{\partial x^1} &= \frac{\partial f_2}{\partial x^2} & \frac{\partial f_0}{\partial x^2} &= -\frac{\partial f_2}{\partial x^1}
\end{align}
which shows the function $z$ is indeed monogenic. For more detail on the relationship of monogenic spinor fields to complex holomorphic functions see Doran and Lasenby \cite[\S 6.3.1]{doran_geometric_2003}.

Since $\grad$ is grade-1, we have $\grad \colon \smoothfields{M}{\G^{\pm}} \to \smoothfields{M}{\G^\mp}$ which yields the direct sum decomposition
\begin{equation}
    \monogenics(M) = \monogenics^+(M) \oplus \monogenics^-(M).
\end{equation}
We can note that each of the graded components of $f_+\in \monogenics^+(M)$ are also harmonic $\Delta \proj{2k}{f_+}=0$. Also, the fact that a product of spinor fields is again a spinor field will make it more convenient to work with $\monogenics^+(M)$ instead of the whole of $\monogenics(M)$. Not only is it more convenient, but to prove the \cref{thm:gelfand}, we only need $\monogenics^+(M)$. 

A key piece of the proof for \cref{thm:gelfand} will be the ability to uniformly approximate monogenic spinor fields on open subsets $U$ by fields defined on $M$. We will cite \cite[theorem 11.7]{calderbank_geometrical_1995} from Calderbank's thesis.
\begin{theorem}[Calderbank]
\label{thm:calderbank}
Let $U$ be an open subset of $M$. Then any monogenic fields on $U$ may be approximated (locally uniformly in all derivatives) by restrictions of monogenic fields on $M$.
\end{theorem}
This will be used repeatedly. Other supplemental information can be found in \cite{booss-bavnbek_dirac_1993} as well. 

From Ryan's \cite[theorem 4]{ryan_clifford_2004}, we have the fact that for $\ball$ a ball in Euclidean space, there exists a power series representation. In fact, one can think of this as an explicit realization of the approximation given by Calderbank which uses the embedding $\ball \subset \R^n \subset S^n$ where $S^n$ is the $n$-sphere. We will construct the power series in \cref{subsec:power_series}. The coefficients, which we define in \cref{eq:coefficients}, will require us to compute integrals of multivector fields which we define next.

In order to integrate, we build differential forms from $k$-vector fields by attaching a measure. Take the coordinate measures (dual basis in the cotangent space) $dx^i$ in local coordinates and multiply by the corresponding reciprocal vector to get \emph{basic directed measures} $d\blade{x}^i = \blade{e}^i dx^i$ (no summation implied), which determine the \emph{$k$-dimensional directed measures}
\begin{equation}
    dX_k \coloneqq \frac{1}{k!} \sum_{i_1 < \cdots < i_k} d\blade{x}^{i_1}\wedge \cdots \wedge d\blade{x}^{i_k} = \frac{1}{k!} \sum_{\mathcal{I}} \blade{E}^{\mathcal{I}^\dagger} dx^{\mathcal{I}}
\end{equation}
An arbitrary differential $k$-form $\alpha_k$ is given by taking a corresponding $k$-vector $A_k$ and contracting along the $k$-dimensional directed measure
\begin{equation}
\alpha_k = A_k \contract dX_k^\dagger.
\end{equation}
Specifically, $A_k = \sum_{\mathcal{I}} \alpha^{\mathcal{I}} \blade{E}_{\mathcal{I}}$ is called the \emph{multivector equivalent of $\alpha_k$}. This is a realization of the isomorphism between $\smoothfields{M}{\G}$ and $\Omega(M)$ as $C^\infty(M)$-modules and it can be viewed as an extension of the musical isomorphisms between vectors and 1-forms \cite[chapter 13]{lee_introduction_2012}. The multivector equivalent of the Riemannian volume form $\mu$ is $\pseudoscalar^{-1^\dagger}$ and $\blade{I}(x)$ represents the tangent space $T_x M$. Of course, in the case where $g$ is definite, the equivalent to $\mu$ is simply $\pseudoscalar$. On $\partial M$, we have the \emph{boundary pseuodoscalar} $\blade{I}_\partial$ and dual to this the boundary normal vector field $\normal = \pseudoscalar_\partial^\perp$. As on $M$, the boundary pseudoscalar $\blade{I}_\partial$ is the multivector equivalent of the boundary area form $\mu_\partial$. From this point forward, we work solely with multivector fields and contract with directed measures to integrate. We now realize the action of $\grad \wedge$ and $\grad \contract$ as the equivalents of $d$ and $\delta$ by
\begin{align}
d \alpha_k &= (\grad \wedge A_k) \contract dX_{k+1}^\dagger, & \delta \alpha_k &= (\grad \contract A_k)\contract dX_{k-1}^\dagger.
\end{align}

One beautiful result in Clifford analysis is the generalization of the Cauchy integral formula. Details for various cases are in our standard sources \cite{doran_geometric_2003,hestenes_clifford_1984, calderbank_geometrical_1995, booss-bavnbek_dirac_1993}. For $\R^n$ with $n\geq 2$, we define the vector-valued function
\begin{equation}
\label{eq:greens_function}
\blade{G}(\blade{x})\coloneqq \frac{1}{\omega_n} \frac{\blade{x}}{|\blade{x}|^n}
\end{equation}
where $\omega_n$ is the area of the unit sphere in $\R^n$ and the use of the bold $\blade{x}$ indicates that we treat this point in space as a vector. For this work, the bold variable $\blade{x}$ will be a clear way to distinguish when we are embedded in $\R^n$, so it is important to keep track of this. The function $\blade{G}$ is the Green's function of the Dirac operator since
\begin{equation}
\label{eq:fundamental_solution}
\grad \blade{G}(\blade{x}) = \delta_{\blade{x}},
\end{equation}
where $\delta_{\blade{x}}$ is the Dirac mass located at $\blade{x}\in \R^n$. There is also a solution in the case $n=1$, but it assumes a different form and we do not need it here.

Define the \emph{Cauchy kernel} $\blade{G}(\blade{x}'-\blade{x})\normal(\blade{x}')$ by translation and multiplication by the outward normal $\normal$ on the boundary $\partial M$. For compact $M\subset \R^n$ and monogenic field $f\in \monogenics(M)$, we convolve the boundary value of $f$ with the Cauchy kernel to arrive at the \emph{Cauchy integral formula}
\begin{equation}
\label{eq:cauchy_integral}
f(\blade{x}) = (-1)^{n} \int_\boundary \blade{G}(\blade{x}'-\blade{x}) \blade{\nu}(\blade{x}') f(\blade{x}') \mu_\partial(\blade{x}').
\end{equation}
Hence, we have a method for uniquely determining a monogenic field $f$ from the boundary values $f\vert_\boundary$. A more general notion of the Cauchy integral exists for arbitrary manifolds with boundary (see \cite{calderbank_geometrical_1995}), but it seems this also depends on a choice of embedding into a closed manifold. In the case of our Green's function $\blade{G}$, this corresponds to the embedding $M\subset \R^n \subset S^n$ \cite[proposition 9.10]{calderbank_geometrical_1995}.

\subsection{Currents}
\label{sec:fields_and_currents}

For the remainder, let $M$ be a smooth oriented compact Riemannian manifold of dimension $n$ with boundary $\partial M$ and positive definite metric tensor $g$. Let $\blade{I}$ be the unit pseudoscalar on $M$ which is the multivector equivalent of the Riemannian volume form $\mu$, $\pseudoscalar_\partial$ the boundary pseuodoscalar which is the multivector equivalent of the Riemannian boundary area form $\mu_\partial$ and dual to $\blade{\nu}$ the boundary normal field. The space $\contfields{M}{\G}$ comes with the \emph{uniform norm}
\begin{equation}
\|f\| \coloneqq \sup_{x\in M} |f(x)|.
\end{equation}
Recall that since $g$ is positive definite, at any point $x\in M$ we have that $|f(x)|^2=(f(x),f(x))$ is nothing but the Euclidean vector norm on $\R^{2n}$, so $\|f\|$ really is a norm on $\contfields{M}{\G}$. Furthermore:

\begin{proposition}
If $M$ is a compact Riemannian manifold, then the space $\contfields{M}{\G}$ is a (real) $C^\ast$-algebra with involution $\dagger$.
\end{proposition}
\begin{proof}
Note that $\G$ is a real $2^n$ dimensional Banach space with the multivector inner product. Since $M$ is a compact Hausdorff space, it follows that the space $\contfields{M}{\G}$ is a Banach space (see \cite{saab_integral_1991}). Taking $f,g \in \contfields{M}{\G}$, at each point
\begin{align}
(fg,fg)  = (gg^\dagger, f^\dagger f),
\end{align}
since $\dagger$ is the adjoint. Using the Cauchy--Schwarz inequality
\begin{align}
|fg|^2 = (fg,fg) \leq (f^\dagger f, f^\dagger f) (gg^\dagger,gg^\dagger) = |f|^2 |g|^2.
\end{align}
The last equality follows from taking an orthonormal basis $\blade{e}_i$ at any $\G_xM$ and forming the orthonormal vector basis blades (versors) $\blade{E}_\mathcal{I}$ and putting $f=\sum_{\mathcal{I}} f^\mathcal{I} \blade{E}_\mathcal{I}$. Then we have
\begin{align}
f^\dagger f = \sum_{\mathcal{I}}\sum_{\mathcal{J}} f^\mathcal{I} f^\mathcal{J} \blade{E}_\mathcal{I}^\dagger \blade{E}_\mathcal{J}
\end{align}
from which we see that
\begin{align}
(f^\dagger f,f^\dagger f)=|f^\dagger f|^2 = \sum_{\mathcal{I}}\sum_{\mathcal{J}}\left(f^\mathcal{I} f^\mathcal{J}\right)^2
\end{align}
and finally
\begin{align}
\left(|f|^2\right)^2 = \left( \sum_{\mathcal{I}} \left(f^\mathcal{I}\right)^2\right)^2 = \sum_{\mathcal{I}}\sum_\mathcal{J} \left(f^\mathcal{I}\right)^2 \left(f^\mathcal{J}\right)^2
\end{align}
which implies that $|f^\dagger f|=|f|^2$. Taking suprema, it follows that $\|fg\|\leq \|f\|\|g\|$ which shows $\contfields{M}{\G}$ is a real Banach algebra.

For $f,g\in \contfields{M}{\G}$ and $\lambda \in \R$ we have that
\begin{equation}
(f+g)^\dagger=f^\dagger + g^\dagger, \quad (\lambda f)^\dagger = \lambda^\dagger f^\dagger = \lambda f^\dagger, \quad f^{\dagger \dagger}=f, \quad (fg)^\dagger = g^\dagger f^\dagger,
\end{equation}
by definition and at each point
\begin{align}
    |f^\dagger f|=|f|^2
\end{align}
as shown before. By taking suprema, $\|f^\dagger f\| = \|f\|^2$ which shows $\contfields{M}{\G}$ is a real $C^\ast$-algebra.
\end{proof}

We topologize the $C^\ast$-algebra $\contfields{M}{\G}$ with the uniform norm topology. Dually, we construct $\G$ valued functionals on this space which we call currents \emph{\`a la} de Rham.

\begin{definition}
The space of \emph{$\G$-currents} is
\begin{equation}
\contcurrents{\G}{\G} \coloneqq \{T \colon \contfields{M}{\G} \to \G ~\vert~ T \textrm{ is continuous}\}
\end{equation}
Given a subalgebra $\algebra{}\subset \G$, we have the \emph{$\mathcal{A}$-currents} 
\begin{equation}
\contcurrents{\mathcal{A}}{\G} = \{T \colon \contfields{M}{\G} \to \G ~\vert~ T \textrm{ is continuous}\}.
\end{equation}
\end{definition}
The $\G$-currents are given the \emph{weak-$\ast$ topology}, i.e., the coarsest topology on $\contcurrents{\G}{\G}$ where point evaluation of fields is continuous. Specifically, for $x \in M$, the Dirac mass $\G$-current $\delta_x \in \contcurrents{\G}{\G}$ defined by $\delta_x[f]=f(x)$ for $f\in \contfields{M}{\G}$ is continuous. The $\mathcal{A}$-currents inherit the subspace topology.

Since the target $\G$ of the $\G$-currents is itself a $C^\ast$-algebra and a $\G$-module, we expect some currents to respect these algebraic structures. For example, $\contfields{M}{\G^+}$ is a $\G^+$-bimodule. Given a $\algebra{} \subset \G^+$, $\G^+$ and $\contfields{M}{\G^+}$ are both $\algebra{}$-modules and Banach algebras.

\begin{definition}
Let $\algebra{}\subset \G$ be a subalgebra and let $T \in \contcurrents{\G}{\G}$ be a $\G$-current. We say that $T$ is \emph{right $\algebra{}$-linear} if it is a right $\algebra{}$-module homomorphism
\begin{equation}
    T[f\alpha + g] = T[f]\alpha + T[g]
\end{equation}
for $\alpha \in \algebra{}$ and $f,g\in \contfields{M}{\G}$. Furthermore, we say that $T$ is \emph{multiplicative on $\algebra{}$} if it is an $\R$-algebra homomorphism
\begin{equation}
    T[pq] = T[p]T[q]
\end{equation}
for $p,q\in \contfields{M}{\algebra{}}$. Finally, a current $T$ is \emph{grade preserving} if for $h\in C^0(M;\G^k)$ we have $T[h]\in \G^k$.
\end{definition}

The set of grade preserving linear multiplicative currents are the most useful for us. It is worth remarking that currents as defined here provide an ample setting for further study. There are plenty of tweaks that could be interesting. One such example would be the subset of the de Rham currents (dual to the $C^\infty$-smooth forms) given by the $\R$-currents $\contcurrents{\R}{\G}$. Here, we will make choices that allow us to generalize the classical Gelfand result.

\subsection{Subsurface fields}

The algebra $\contfields{M}{\G}$ is not commutative in general and the space $\monogenics(M)$ is not an algebra. This poses a direct issue for a straightforward generalization of the Gelfand representation and Belishev's 2-dimensional boundary control method \cite{belishev_calderon_2003}, but it does not thwart the effort completely since Belishev and Vakulenko manage to build a 3-dimensional version \cite{belishev_algebras_2020}. Extending this approach, we will use insight on axial fields from those two authors but make the change to think not of an axis, but of a plane. Of course, in $\R^3$ a plane and axis are dual, but when we extend beyond dimension-3, we will be required to use planes. If $S$ is dimension 2, then $\monogenics^{+}(S)$ is a copy of the commutative algebra of holomorphic functions. Intuitively, we can build commutative Banach algebras of monogenic fields for surfaces in $M$.

Let $U\subset M$ be a geodesically convex region, i.e., that all points $x\in U$ are connected with unique shortest paths. Let $\bivector(x)$ be a unit 2-blade in $\G_xU$ for some $x\in U$. Since $U$ is convex, there exists a shortest geodesic between all points in $U$ which allows us to parallel transport $\bivector(x)$ to build a unit 2-blade field $\bivector \in \smoothfields{U}{\G^2}$. Then, at all points in $U$, we have a projection  $\projection_{\bivector}$ onto $\bivector(y)$ in each geometric tangent space $\G_yU$.

\begin{definition}
Let $U$ and $\bivector$ be as before, then a continuous spinor field $f \in \contfields{U}{\G^+}$ satisfying
    \begin{equation}
    f_+ = \projection_{\bivector} \circ f_+
    \end{equation}
is a \emph{subsurface spinor field on $U$}.
\end{definition}

The definition for a subsurface spinor field on $U$ requires that $f_+ = \projection_{\bivector} \circ f_+$ which means that we can put $f_+=f_0+f_2 \bivector$ where $f_0,f_2 \in \contfields{U}{\R}$.

\begin{definition}
Let $U$ and $\bivector$ be as before, then \emph{the space of monogenic subsurface spinors on $U$} is
\begin{equation}
\algebra{\bivector}(U) = \{f_+ \in \contfields{U}{\G^+} ~|~ f_+ = \projection_{\bivector} \circ f_+,~ \grad f_+ = 0\}.
\end{equation}
The \emph{collection of all monogenic subsurface spinors on $U$} is
\begin{equation}
\algebra{}(U) = \{f_+ \in \algebra{\bivector}(U) ~\vert~ \textrm{$\bivector$ parallel transported from $\bivector(x)\in \G_xU$, ~ $\forall \bivector(x)\in \G_xU$} \}.
\end{equation}
\end{definition}

\begin{proposition}
Let $U$ and $\bivector$ be as before, then the space $\algebra{\bivector}(U)$ is a commutative Banach algebra.
\end{proposition}
\begin{proof}
Note that multiplication of two fields $f=f_0+f_2 \bivector$ and $g=g_0+g_2 \bivector$ (dropping the subscripted $+$ on $f$ and $g$ momentarily for clarity) in $\algebra{\bivector}(U)$ is commutative and given pointwise by the familiar complex multiplication
\begin{equation}
\label{eq:axial_multiplication}
fg = f_0 g_0 - f_2 g_2 + \bivector (f_0 g_2 + f_2 g_0) = gf.
\end{equation}
Using the overdot notation to say which field we are taking derivatives of, we find commutivity gives us algebraic closure since
\begin{align}
\grad(fg) &= \grad fg+\dot{\grad}f \dot{g} &&\textrm{by the Leibniz rule}\\
&=0 + \grad g f && \textrm{since $f$ is monogenic and $fg=gf$}\\
&=0 && \textrm{since $g$ is monogenic.}
\end{align}
Since $\algebra{\bivector}(U)$ is a subalgebra of $\contfields{U}{\G}$, it is a commutative Banach algebra.
\end{proof}

This construction provides a notion of complex functions that are nested in multivector fields on any manifold of dimension $n\geq 2$. In the case $n=1$, no such fields exist and it is exactly in the 2-dimensional Euclidean case that the complex-valued functions are just the spinor fields themselves and the unit 2-blade field is the tangent pseudoscalar to the surface. The special case of monogenic subsurface spinor fields serve as a realization of complex holomorphic functions inside the more general spinor fields. If we take $\bivector =\blade{e}_1\blade{e}_2$, then we have the Cauchy--Riemann equations from $\grad f_+ = 0$ via \cref{eq:cauchy_riemann_equations}.

For example, take the case where $M$ is a compact region of $\R^n$ with the Euclidean metric. Then $M$ itself is compactly contained inside of some ball $\ball$ which is convex. The set of bivectors is parameterized by $\bivector \in \Grassmannian{2}{n}$ (i.e., the possible coordinate planes) and for each such $\bivector$ we can consider $\algebra{\bivector}(M)$ as a restriction of $\algebra{\bivector}(\ball)$ via \cref{thm:calderbank}. Each unit 2-blade decomposes into two orthogonal unit vectors. Let $\bivector = \blade{v}\blade{w}$ where $\blade{v}$ and $\blade{w}$ are a pair of orthogonal unit vectors and consider the monogenic subsurface field $z\colon M \to \planespinors \subset \G^+$ defined by
\begin{equation}
\label{eq:z}
z(\blade{x}) \coloneqq  \projection_{\bivector}(\blade{v} \blade{x}).
\end{equation}
It is immediately clear that $z = \projection_{\bivector} \circ z$ and in applying the Hodge--Dirac operator
\begin{align}
\grad z &= \grad (\blade{x}\cdot \blade{v}) + \grad (\blade{x}\cdot \blade{w})\bivector\\
    &= 0.
\end{align}
We can define such a function $z$ for any choice of $\bivector$ and construct new functions from polynomials in these variables. The notation $z$ should serve as a reminder of the connection to complex analysis and one may consider $\blade{v}$ as the real axis and $\blade{w}$ as the imaginary axis. The behavior of fields on an arbitrary convex $U$ inside an arbitrary compact $M$ is identical.

\section{A Clifford-Algebraic Gelfand Theorem}

\subsection{Spinor Spectrum}

Through Belishev's generalization of Gelfand's classical result, surfaces are determined up to conformal equivalence by the spectrum (or maximal ideal space) of $\monogenics^+(S)$. The naive generalization would be to seek this out in $\monogenics^+(M)$, but, again, this space is not an algebra! Maximal ideals can also be identified with multiplicative functionals and this allowed Belishev and Vakulenko to achieve their 3-dimensional result. We follow suit with multiplicative linear currents.

\begin{definition}
    The \emph{spinor spectrum} $\characters(M) \subset \contcurrents{\G^+}{\G^+}$ is the set of nonzero grade preserving right linear currents that are multiplicative over the collection of all subsurface spinor algebras $\algebra{}(U)$,
    \begin{multline}
        \characters(M) \coloneqq \{ \delta \neq 0 \in \contcurrents{\G^+}{\G^+} ~\vert~ \textrm{$\delta$ grade preserving,}\\
\delta(fg+h\alpha) = \delta(f)\delta(g)+\delta(h)\alpha, ~ \forall f,g,h \in \algebra{}(U), ~\alpha \in \G^+\},
    \end{multline}
    and we refer to the elements as \emph{spin characters}.
\end{definition}

One choice of spin character is point evaluation: if $\delta$ is defined on $f_+\in \monogenics^+(M)$ by $\delta(f_+)=f_+(x^\delta)$ for some $x^\delta \in M$, then it follows that $\delta \in \characters(M)$. This shows that $M$ injects into $\characters(M)$ by the map $x \mapsto \delta_x$ where $\delta_x[f]=f(x)$. We will find that (at least for embedded $M$) characters defined by point evaluation are the only elements of $\characters(M)$. This shows the inclusion is surjective. In fact, the main result is that this map is a homeomorphism.

\begin{theorem}
\label{thm:gelfand}
Let $M$ be a compact region in $\R^n$. For any $\delta \in \characters(M)$, there is a point $\blade{x}_\delta \in M$ such that $\delta(f) = f(\blade{x}_\delta)$ for any $f_+\in \monogenics^+(M)$ a monogenic field. Given the weak-$\ast$ topology on the space of $\G$-currents, the map
\[
\Gamma \colon \characters(M) \to M, \quad \delta \mapsto \blade{x}_\delta
\]
is a homeomorphism. The Gelfand transform $\monogenics^+(M) \to \contfields{\characters(M)}{\G^+}$ given by $\widehat{f_+}(\delta) = \delta[f_+]$ is an isometric isomorphism onto its image so that $\monogenics^+(M)\cong \widehat{\monogenics^+(M)}$.
\end{theorem}
We prove \cref{thm:gelfand} in the following steps:
\begin{enumerate}[i.]
\item Utilize a power series representation for elements in a ball $\ball$ which shows that the monogenic polynomials $\monogenics^\mathcal{P}(\ball)$ are dense in $\monogenics^+(M)$.
\item Build the elements of this series from homogeneous polynomials in variables of the form $z$ (i.e., \cref{eq:z}). Using the fact that the spin characters are multiplicative over the collection $\algebra{}(M)$, continuous, and $\G^+$-linear we show that it suffices to determine the action $\delta[z]$ for $\delta \in \characters(M)$.
\item Determine that the action $\delta[z]$ is point evaluation at some point $\blade{x}_\delta \in \R^n$ by looking at the algebraic relationships between the variables $z$ and combining this with the multiplicativity of $\delta$.
\item Construct a carefully selected sequence of monogenic fields on $M$ and use continuity of $\delta$ to show that $\blade{x}_\delta \in M$.
\end{enumerate}
The correspondence between $\delta \in \characters(M)$ is then clear and the homeomorphism follows by choice of the weak-$\ast$ topology. The fact that the Gelfand transform is an isometry follows directly from the fact that each character corresponds to point evaluation.

\subsection{Power series}
\label{subsec:power_series}

Take the standard orthonormal basis fields $\blade{e}_1,\dots,\blade{e}_n$ for $\R^n$ and define the functions $z_{ij} = x^j - x^i \bivector_{ij}$ where $x^i$ are the coordinate functions corresponding to our basis vector fields and where $\bivector_{ij}\coloneqq \blade{e}_i \blade{e}_j$. This takes \cref{eq:z} and multiplies by $\bivector_{ij}$ to match Ryan \cite{ryan_clifford_2004} and, in effect, this is just replacing $z$ with $\mathbf{i}z$. Note that Ryan's use of $\blade{e}_i^{-1}$ become unnecessary due to our choice of a positive definite quadratic form. Identifying $\bivector_{ij}$ with its plane, note that for any compact region $M\subset \R^n$ each $z_{ij} \in \algebra{\bivector_{ij}}(M)$. Fix a natural number $k\geq 0$ and natural numbers $k_j$ to form the tuple $\vec{k}=(k_2,\dots,k_n)$ such that $k_2+\cdots +k_n=k$ with $k_j\geq 0$. This is often called a multi-index with absolute value $|\vec{k}|=k$. The set of all multi-indices of absolute value $k$ is of size ${n-2 + k \choose n-2}$. For example, we can write down a degree-$k$ polynomial in terms of the monomial variables $z_{ij}$ based on a multi-index $\vec{k}$ by
\begin{equation}
\label{eq:ex_hom_poly}
z_{12}^{k_2}(\blade{x}) z_{13}^{k_3}(\blade{x}) \cdots z_{1n}^{k_n}(\blade{x}).
\end{equation}
But, ordering does matter. To build the \emph{homogeneous monogenic degree $k$ polynomials} we sum over permutations $\sigma$ which rearrange the order in which we write the monomials but keep the total powers of each monomial the same throughout
\begin{equation}
\label{eq:homogeneous_monogenic_polynomials}
        p_{\vec{k}}(\blade{x}) = \frac{1}{k!} \sum_{\sigma}z_{1 \sigma(1)}(\blade{x}) \cdots z_{1 \sigma(k)}(\blade{x}),
\end{equation}
where $\sigma(j) \in \{2,\dots,n\}$ permutes the monomials without rearrangement. To reiterate, monomials that appear in the summand of \cref{eq:homogeneous_monogenic_polynomials} with the powers given by $\vec{k}$ are just reordered from what we see in \cref{eq:ex_hom_poly} based on $\sigma$ which is why we do not see $\vec{k}$ explicitly appear on the right hand side of \cref{eq:homogeneous_monogenic_polynomials}. Note that this is necessary since the monomials do not commute with one another. Ryan \cite[Proposition 1]{ryan_clifford_2004} shows each of these polynomials is monogenic and linearly independent. We remark that the polynomials $p_{\vec{k}}$ are homogeneous in the elements $z_{ij} \in \algebra{\blade{B}_{ij}}(M)$. 

As examples, take $n=3$ and $k=2$ with $k_2=2$ and $k_3=0$ so that the multi-index is $\vec{k}=(2,0)$. Then in coordinates $\blade{x} = (x^1,x^2,x^3)$
\begin{align}
p_{(2,0)}(\blade{x}) &= \frac{1}{2!} \sum_{\sigma} z_{1\sigma(1)} z_{1\sigma(2)}\\
		&= \frac{1}{2!} z_{12}(\blade{x}) z_{12}(\blade{x})\\
        &=\frac{1}{2!} (x^2-x^1 \blade{e}_{1}\blade{e}_2)^2.
\end{align}
We can see that given our choice of $\vec{k}$, there is only one choice of $\sigma$, i.e., the $\sigma$ such that $\sigma(1)=2$ and $\sigma(2)=2$. On the other hand if we take the multi-index $\vec{k}=(1,1)$, then
\begin{align}
p_{(1,1)}(x_1,x_2,x_3) &= \frac{1}{2!} \sum_{\sigma} z_{1\sigma(1)} z_{\sigma(2)}\\
		&= \frac{1}{2!}\left( z_{12}(\blade{x}) z_{13}(\blade{x}) + z_{13}(\blade{x}) z_{12}(\blade{x}) \right)\\
        &= \frac{1}{2!}\left((x^2-x^1 \blade{e}_{1}\blade{e}_2)(x^3-x^1 \blade{e}_{1}\blade{e}_3)+(x^3-x^1 \blade{e}_1\blade{e}_3)(x^2-x^1 \blade{e}_{1}\blade{e}_2)\right).
\end{align}
Again, our choice of $\vec{k}$ allowed for two choices of $\sigma$ that were not repetitive: first $\sigma(j)=j$ and the other $\sigma(1)=2$ and $\sigma(2)=1$. Furthermore, working out the details of $\grad p_{(k_2,k_3)}$ shows the necessity of summing over permutations in order to ensure that each is monogenic. The collection of all such polynomials for all multi-indices is the set of \emph{monogenic polynomials}
\begin{equation}
    \monogenics^\mathcal{P}(M) = \left.\left\{\sum_{k=0}^N \left(\sum_{\substack{{\vec{k}} \\ {|\vec{k}| = k}}}p_{\vec{k}}(\blade{x}) a_{\vec{k}}\right) ~\right\vert~ N\in \mathbb{N}, ~ a_{\vec{k}} \in \G_n\right\}
\end{equation}

The use of multi-index notation is also to facilitate taking higher order partial derivatives by defining
\begin{equation}
    \nabla^{\vec{k}} \coloneqq \frac{\partial^k}{\partial x_2^{k_2} \partial x_3^{k_3} \cdots \partial x_n^{k_n}}.
\end{equation}
In the case of a smooth manifold, the partial derivatives can be replaced with their corresponding covariant derivatives if the need should arise. Next, \cref{lem:density}, \cref{cor:power_series_for_regions}, and \cref{prop:local_power_series} show that for arbitrary $M$, $\monogenics(M)$ are locally uniformly approximated by monogenic polynomials.

\begin{lemma}
\label{lem:density}
Let $\ball$ be a compact ball in $\R^n$, then the space $\monogenics^\mathcal{P}(\ball)$ is dense in $\monogenics(\ball)$.
\end{lemma}
\begin{proof}
Without loss of generality, suppose $\ball$ is centered at the origin. Then let $f\in \monogenics(\ball)$ and define the coefficients $a_{\vec{k}}\in \G_n$ by
\begin{equation}
\label{eq:coefficients}
a_{\vec{k}} = \int_{\partial \ball} \nabla^{\vec{k}} \boldsymbol{G} \normal f \mu_{\partial},
\end{equation}
where $\boldsymbol{G}$ is the Green's function for the Hodge--Dirac operator in $\R^n$. By \cite[theorem 4]{ryan_clifford_2004}, we have
\begin{equation}
        f(\blade{x}) = \sum_{k=0}^\infty \left(\sum_{\substack{{\vec{k}} \\ {|\vec{k}| = k}}} p_{\vec{k}} (\blade{x}) a_{\vec{k}} \right),
\end{equation}
which converges uniformly to $f$ for points $\blade{x}\in \interior \ball$.
\end{proof}

But, as stated previously, we have \cref{thm:calderbank} which tells us that for open subsets in $\ball$ we can uniformly approximate monogenic fields on those subsets. We apply this fact to get the following corollary.

\begin{corollary}
\label{cor:power_series_for_regions}
Let $M\subset \R^n$ be a compact region. Then there exists $\ball$ such that $\monogenics(\ball)$ are dense in $\monogenics(M)$.
\end{corollary}
\begin{proof}
Since $M$ is compact in $\R^n$ there exists a ball $\ball$ such that the closure of $M$ is contained in $\ball$. Then by \cref{thm:calderbank}, any monogenic fields on $M$ can be uniformly approximated by monogenic fields in $\monogenics(\ball)$, and we have our result by \cref{lem:density}, 
\end{proof}
\begin{proposition}
\label{prop:local_power_series}
Let $M$ be an $n$-dimensional compact Riemannian manifold and let $f\in \monogenics(M)$. Then $f$ admits a local power series representation over finitely many open subsets.
\end{proposition}
\begin{proof}
Take $f\in \monogenics(M)$, let $(U,\varphi)$ a local coordinate chart such that $U\subset M$ is an open convex region and $\varphi(U) \subset \ball \subset \R^n$ where $\ball$ is some closed ball in $\R^n$. Then $f \circ \varphi \in \monogenics(\varphi(U))$ and by \cref{lem:density} and \cref{cor:power_series_for_regions}, $f\circ \varphi$ admits a power series representation. Since $M$ is Riemannian there exists a finite covering of $M$ by convex regions and this gives us a local power series representation over finitely many open sets.
\end{proof}

Following the details of the above proofs for a surface $S$ yields the fact that holomorphic functions on a surface admit local power series representations. In this case, take $x^1, x^2$ as local isothermal coordinates and define $z=x^2 - x^1\bivector$. Then for $f_+\in \monogenics^+(S)$, we have the local power series $f_+(z)=\sum_{k=0}^\infty z^k a_k $ where $a_k \in \planespinors$. 

\begin{remark}
It is important to note that if $f_+\in \monogenics^+(M)$, the local power series has coefficients $a_{\vec{k}} \in \G^+$ which you can see by \cref{eq:coefficients}.
\end{remark}

\subsection{Characters}

For $M$ a compact region embedded in $\R^n$ and $f_+\in \monogenics^+(M)$, we can see that for $\delta \in \characters(M)$ 
\begin{align}
\delta\left[f\right] &= \sum_{k=0}^\infty \left( \sum_{\vec{k}} \delta[p_{\vec{k}}] a_{\vec{k}} \right)
\end{align}
by continuity and right $\G^+$-linearity of $\delta$ since $a_{\vec{k}}\in \G^+$. On each monogenic polynomial,
\begin{align}
\delta(p_{\vec{k}}) &= \frac{1}{k!} \sum_{\sigma}\delta\left[z_{1\sigma(1)} \right] \cdots \delta\left[z_{1\sigma(k)}\right],
\end{align}
by multiplicativity over $\algebra{}(M)$. Hence, the action of $\delta$ is completely determined by the action on each $z_{ij}$.

\begin{proposition}
\label{prop:surface_spinors_map_to_plane_spinors}
Let $M$ be a compact manifold embedded in $\R^n$ and $\bivector$ a unit 2-blade field that is a parallel translation of a coordinate plane. Then for any $\delta \in \characters(M)$ we have $\delta(\algebra{\bivector}(M)) = \planespinors$.
\end{proposition}
\begin{proof}
Since $\delta$ is grade preserving, it must be the case that $\delta[z]\in \G^{0\oplus 2}$. Since $\delta$ is an algebra morphism, $\delta[\algebra{\bivector}(M)]\subset \G^{0\oplus 2}$ is commutative subalgebra. Using linearity as well, $\delta[\alpha + \beta \bivector]=\delta[1](\alpha + \beta \bivector) = \alpha + \beta \bivector$ for $\alpha, \beta \in \R$. Hence $\planespinors \subset \delta[\algebra{\bivector}(M)]$. If $\tilde{\bivector} \in \delta[\algebra{\bivector}(M)]$ commutes with $\bivector$, these bivectors must not intersect as subspaces except at zero which yields the 4-vector $\bivector \tilde{\bivector}\notin \G^{0\oplus 2}$. This contradicts the grade preservation of $\delta$ and thus $\delta[\algebra{\bivector}(M)]=\planespinors$.
\end{proof}

Next, we show that the characters $\characters(M)$ correspond to evaluation at some point in $\R^n$.

\begin{lemma}
\label{lem:correspondence}
Let $M$ be a compact region in $\R^n$ and $\delta \in \characters(M)$. Then $\delta(z)=z(\blade{x}_\delta)$ for some $\blade{x}_\delta \in \R^n$.
\end{lemma}
\begin{proof}
Take $\delta \in \characters(M)$ and the coordinate planes $\bivector_{ij}$ and the corresponding $z_{ij}$. Applying $\delta$ to $z_{ij}$ yields $\delta[z_{ij}] = \alpha_{ij} + \beta_{ij} \bivector_{ij}$ with $\alpha_{ij}, \beta_{ij} \in \R$ by \cref{prop:surface_spinors_map_to_plane_spinors} and we will collect all $\alpha_{ij}$ and $\beta_{ij}$ into matrices $\alpha$ and $\beta$ respectively. Then, since
\begin{equation}
\label{eq:z_reciprocal_relationship}
z_{ij} \bivector_{ji} = (x^j-x^i \blade{e}_i\blade{e}_j)\blade{e}_j\blade{e}_i = -z_{ji}
\end{equation}
it follows that
\begin{equation}
\delta[z_{ij} \bivector_{ji}] = \delta[z_{ij}] \bivector_{ji} = -\delta[z_{ji}],
\end{equation}
and hence
\begin{equation}
(\alpha_{ij} + \beta_{ij} \bivector_{ij})\bivector_{ji} = \beta_{ij}+\alpha_{ij}\bivector_{ji} = - \alpha_{ji} - \beta_{ji} \bivector_{ji}.
\end{equation}
Therefore, $\alpha_{ij} = -\beta_{ji}$ for all $i \neq j$. Similarly, for arbitrary $\ell \neq i$ and $\ell  \neq j$ we have
\begin{equation}
\label{eq:z_relationship}
z_{ij} = z_{\ell j} + z_{i\ell} \bivector_{\ell j}
\end{equation}
so
\begin{equation}
\delta[z_{ij}] = \delta[z_{\ell j} + z_{i\ell } \bivector_{\ell j}] = \delta[z_{\ell j}]+\delta[z_{i\ell }]\bivector_{\ell j}.
\end{equation}
Expanding this yields the relationships $\alpha_{ij} = \alpha_{\ell j}$ and $\beta_{ij} = \beta_{i\ell }$ for all $i,j,\ell$.

The relationships $\alpha_{ij} = \alpha_{kj}$ and $\beta_{ij} = \beta_{ik}$ show that both sets of constants $\alpha$ and $\beta$ are given by $n$ numbers since they are constant along one index. Taking this with the relationship $\alpha_{ji} = -\beta_{ij}$ shows that both are determined by the same $n$ numbers which we call $x_\delta^i=\alpha_{ji}=-\beta_{ij}$ for $i=1,\dots, n$, just with swapped index and magnitude. Hence there exists some $\blade{x}_\delta = (x^1_\delta,\dots,x^n_\delta) \in \R^n$ so that $\delta[z_{ij}]=z_{ij}(\blade{x}_\delta)$ since
\begin{equation}
    z_{ij}(\blade{x}_\delta) = x^j_\delta - x^i_\delta \bivector_{ij}.
\end{equation}
\end{proof}

To see that the corresponding point $\blade{x}_\delta$ lies in the given region $M$ for any $\delta$, we use continuity and a singular monogenic spinor field.

\begin{lemma}
\label{lem:identification}
Let $M\subset \R^n$ be a compact region and let $f\in \monogenics^+(M)$ and $\delta \in \characters(M)$. Then $\delta(f)=f(\blade{x}_\delta)$ for some $\blade{x}_\delta \in M$.
\end{lemma}
\begin{proof}
To see that $\blade{x}_\delta \in M$, take $f_0(x)\coloneqq \blade{G}(\blade{x}-\blade{x}_0)\blade{e}_1$ with $\blade{x}_0\not\in M$. Again, $\blade{G}$ is the Green's function for the Hodge--Dirac operator. Then $f_0\vert_M \in \monogenics^+(M)$. By \cref{lem:correspondence} we have some $\blade{x}_\delta\in \R^n$ such that
\begin{align}
\delta(f_0\vert_M)=f_0\vert_M(\blade{x}_\delta).
\end{align}
Take a sequence $\blade{x}_n \to \blade{x}_\delta$ and suppose for a contradiction that $\blade{x}_\delta \notin M$ and each $\blade{x}_n \notin M$. Then this defines a sequence of functions $f_n(\blade{x}) \coloneqq E(\blade{x} - \blade{x}_n)\blade{e}_1 \vert_M \in \monogenics^+(M)$ and the sequence converges uniformly to a monogenic function $\lim_{n\to \infty} f_{n}(\blade{x}) = \blade{G}(\blade{x}-\blade{x}_\delta)\blade{e}_1$. By continuity of $\delta$,
\begin{align}
\lim_{n\to \infty} \delta(f_n) = \lim_{n\to \infty} f_n(\blade{x}_\delta),
\end{align}
which does not converge due to the singularity at $\blade{x}_\delta$ which contradicts the fact that the limit converges to a monogenic function. Hence, it must be that $\blade{x}_\delta \in M$.
\end{proof}

One practical reason behind working with regions of $\R^n$ is that there are clear choices of functions to use to probe whether a point is in the region or not. For an arbitrary $n$-dimensional manifold we cannot guarantee an embedding into $\R^n$ and the technique using the Cauchy kernel $\blade{G}$ fails and a version of \cref{lem:identification} for arbitrary compact manifolds is not obvious. Likewise, \cref{lem:correspondence} could be viewed as a local result for characters on a coordinate patch, but it is not clear how the restriction of a character to local coordinate patches behaves. Nonetheless, we arrive at the proof for the main theorem.

\begin{proof}[Proof of \cref{thm:gelfand}]
Fix $M$ a compact region of $\R^n$. It is clear that the map $M \to \characters (M)$ is an embedding by mapping a point $\blade{x}\in M$ to $\delta_{\blade{x}} \in \characters(M)$ (inverse of $\Gamma$). Then, by \cref{lem:identification}, any $\delta \in \characters(M)$ corresponds to $\blade{x}_\delta \in M$ showing the reverse inclusion. Hence the sets are in bijection via $\Gamma$ and under the weak-$\ast$ topology, $\Gamma$ is also continuous and hence we have the homeomorphism $M\cong \characters(M)$.

To see that the Gelfand transform $\monogenics^+(M) \to \contfields{\characters(M)}{\G^+}$ is an isometry, we note that
\begin{equation}
\|\hat{f}\| = \sup_{\delta \in \characters(M)} |\hat{f}(\delta)| = \sup_{\delta \in \characters(M)} |\delta[f]| = \sup_{\blade{x}_\delta \in M} |f(\blade{x}_\delta)| =\|f\|.
\end{equation}
Hence, we have our theorem.
\end{proof}

\subsection{Further results and discussion}

Though the behavior of characters on regions has been determined, it is still an open question whether \cref{thm:gelfand} can be extended to arbitrary $n$-dimensional compact Riemannian manifolds with boundary. This extension is not immediately obvious, but there is more to be said that may assist the general case in the future. Again using motivation from complex analysis, $\monogenics^+(M)$ retains some necessary features but others are missing.

\subsubsection{Stone--Weierstrass}
Firstly, let us prove a Stone--Weierstrass result showing the density of closure of the monogenic spinor fields in the space of continuous spinor fields. The proof of the theorem will require the following lemma.

\begin{lemma}
If $M$ be a compact Riemannian manifold with boundary, then the space $\overline{\monogenics^+(M)}$ separates points.
\end{lemma}
\begin{proof}
Let $x,y\in \interior M$ be distinct points. We want to construct some field $f \in \monogenics^+(M)$ such that $f(x)\neq f(y)$. Since $M$ is compact, there exists a shortest path $\gamma \colon [0,1] \to M$ between $x$ and $y$ and moreover by \cite{alexander_geodesics_1981} this path is $C^1$ since both $M$ and $\partial M$ are $C^\infty$. Since $\gamma$ must always be $C^1$, $\gamma$ has a well-defined tangent vector at each $t$ and a well-defined normal space $N_{\gamma(t)}\gamma$ which is orthogonal to the tangent vector $\blade{\dot{\gamma}}(t)$.

Since $M$ is compact, for all $t$ the injectivity radius of the exponential map at $\gamma(t)$ is bounded from below by some $\epsilon > 0$. Hence, we can construct a tube $\mathbb{T}_\gamma$ about $\gamma$ by taking $\mathbb{T}_\gamma \coloneqq \gamma \times \mathbb{D}_\epsilon$ where $\mathbb{D}_\epsilon(t)$ is the image under the exponential map of the disk of radius $\epsilon$ in the normal space at $N_{\gamma(t)}\gamma$. Any point $\tilde{x} \in \mathbb{T}_\gamma$ is given uniquely by coordinates $(t,\blade{v})$ where $\blade{v}\in N_{\gamma(t)} \gamma$. Given a unit 2-blade $\bivector(x) \in \G_x M$, we can parallel translate $\bivector(x)$ to a unit 2-blade $\bivector(\tilde{x})$ any point $\tilde{x} \in \mathbb{T}_\gamma$ by parallel translation along $\gamma$ and then parallel translation in the normal direction. This builds unit 2-blade field $\bivector$ on $\mathbb{T}_\gamma$.

Then, on $\mathbb{T}_\gamma$, define the field $z\in \algebra{\bivector}(\mathbb{T}_\gamma)$ using the unit 2-blade field $\bivector$ following \cref{eq:z}. Then $z(x)\neq z(y)$ and $z\in \monogenics^+(\mathbb{T}_\gamma)$. Since $\mathbb{T}_\gamma$ is a union of open sets, $\mathbb{T}_\gamma$ is open in $M$, and we can uniformly approximate $z$ by elements of $\monogenics^+(M)$. Taking the uniform limit of these approximations yields a function $f \in \overline{\monogenics^+(M)}$ satisfying $f(x)\neq f(y)$.
\end{proof}

The space $\overline{\monogenics^+(M)}$ is not an algebra, but we can consider the minimal algebra that the space generates. Let $\vee \overline{\monogenics^+(M)}$ represent the minimal algebra generated by $\overline{\monogenics^+(M)}$, then using the previous lemma and a result from Laville and Ramadanoff in their paper on the Stone--Weierstrass theorem for Clifford-valued functions \cite{laville_stone-weierstrass_1996}, we will get the following theorem.

\begin{theorem}
$\vee \overline{\monogenics^+(M)}$ is dense in $\contfields{M}{\G^+}$.
\end{theorem}
\begin{proof}
Since $\overline{\monogenics^+(M)}$ contains 1 and separates points, it is a candidate for the use of Laville and Ramadanoff \cite[theorem 3]{laville_stone-weierstrass_1996}. In order to use their result in all dimensions, we must have that $f\in \overline{\monogenics^+(M)}$ is invariant with respect to the principal involution $f_*$. Since $f$ is a spinor field, $f=\sum_{2k=0}^n f_{2k}$ and 
\begin{equation}
f_*=\sum_{2k=0}^n (-1)^{2k}f_{2k}= \sum_{2k=0}^n f_{2k} = f_*,
\end{equation} 
so $f$ is invariant under the principal involution. 

Matching our notation to Laville's, take a basis $2k$-blade $\blade{E}_\mathcal{I}$ (i.e., $|\mathcal{I}|=2k$ is an ordered list of indices and $\blade{E}_\mathcal{I}$ is given by \cref{eq:basis_blades}), then given $f\in\vee \overline{\monogenics^+(M)}$ we can take $f\mapsto f_\mathcal{I} = (f, \blade{E}^\mathcal{I})$ (see \cref{eq:inner_product_with_basis}) which produces a dense subset $\vee \contfields{M}{\R}_\mathcal{I} \subset \contfields{M}{\R}$ by the classical Stone--Weierstrass theorem. Hence, since
\[
\contfields{M}{\G^+} = \bigoplus_{2k} \contfields{M}{\R}\blade{E}_\mathcal{I}
\]
we conclude that $\vee \overline{\monogenics^+(M)}$ is dense in $\contfields{M}{\G^+}$.
\end{proof}

\subsubsection{Tomography}

As discussed earlier, the work in this paper is heavily motivated by the Boundary Control (BC) method for the inverse tomography problem. The BC method was used in Belishev's proof for the 2-dimensional Calder\'on problem in \cite{belishev_calderon_2003}. Our paper manages to show that we can determine the homeomorphism type of an embedded manifold from the spinor spectrum, but we are missing other key facts that would lead to a solution for the tomography problem. Essentially, we need the following additional facts to use the BC method:
\begin{enumerate}[i.]
\item The Dirichlet-to-Neumann (DN) map determines the space $\trace \monogenics^+(M)$.
\item The boundary trace map $\trace \colon \vee \monogenics^+(M) \to \trace \vee \monogenics^+(M)$ where $f_+ \mapsto f_+ \vert_\boundary$ is an isometric isomorphism of algebras. 
\item The space $\monogenics^+(M)$ determines the metric structure of $M$ up to isometry.
\end{enumerate}
Given the results of this paper alongside items (i) and (ii), we would be able to determine a compact embedded $M$ up to homeomorphism. We can view (iii) as an extension of our result here. Namely, we have determined the homeomorphism type of $M$ from the space $\monogenics^+(M)$, but have yet to gain any metric data.

Let us discuss each of the points above.
\begin{enumerate}[i.]
\item Using the Dirichlet-to-Neumann operator $\Lambda$ on differential forms, Belishev and Sharafutdinov in \cite{belishev_dirichlet_2008} describe an ancillary boundary operator called the \emph{Hilbert transform} $T$. The Hilbert transform is a classical operator in complex analysis and it also appears in Clifford analysis as an operator on the $L^2$-completion of the space $\trace \contfields{M}{\G}$. Two good sources include Brackx and De Schepper's paper \cite{brackx_hilbert_2008} (which specifically concentrates on compact regions of $\R^n$) and Calderbank's thesis \cite{calderbank_geometrical_1995}. 

Given a function on the boundary $\phi\in \trace \contfields{M}{\G}$, the Hilbert transform in Clifford analysis allows one to uniquely recover the complete boundary data of a monogenic field $f$ with $\phi$ as a component of $f\vert_{\partial M}$. In essence, the Hilbert transform yields boundary values of functions conjugate in the generalized Cauchy--Riemann equations given by $\grad f=0$. In fact, if $\phi_k$ is a $k$-vector, the Hilbert transform of $\phi_k$ contains a $k-2$-, a $k$-, and $k+2$-vector component. This is part of why we suggest to look beyond pairwise conjugate forms.

A reasonable question to ask is: are these two Hilbert transform operators related in some way? Moreover, Sharafutdinov and Shonkwiler extend Belishev and Sharafutdinov's Dirichlet-to-Neumann operator to the complete Dirichlet-to-Neumann operator \cite{sharafutdinov_complete_2013}. If it is not possible to relate Belishev and Sharafutdinov's Hilbert transform using the DN operator on forms to the Hilbert transform in Clifford analysis, is there a related operator defined in terms of the complete DN operator that relates to Clifford analysis?

\item This point is essentially given as an open question by Belishev and Vakulenko \cite{belishev_algebras_2017}. Specifically, those two ask whether it is true that the algebras $\vee \monogenics^+(M)$ and $\vee \mathrm{tr}\monogenics^+(M)$ are isometrically isomorphic. At the moment, we have a partial answer: by the Cauchy integral formula, a monogenic field $f\in \monogenics(M)$ is determined by its boundary values therefore the map $\trace \colon \monogenics(M) \to \trace \monogenics(M)$ is an isomorphism of vector spaces and by the weak maximum principle for elliptic operators, it is also an isometry. This of course applies to the spinor subspace $\monogenics^+(M)$. However, it is not clear that the algebras $\vee \monogenics^+(M)$ and $\vee \mathrm{tr} \monogenics^+(M)$ are isomorphic. In Clifford analysis, we have the \emph{Hardy space} $\hardy(M)$ as the $L^2$-completion of $\trace \monogenics(M)$ which is studied in both the previously referenced papers \cite{brackx_hilbert_2008,calderbank_geometrical_1995}. Perhaps there is more knowledge about $\hardy(M)$ that could assist in finding a proof of fact (ii).

\item There is likely geometric content inside the spinor spectrum and this could lead to determining, at the very least, the conformal class of the metric. First, it is widely known that the Hodge--Dirac operator is conformally invariant \cite{calderbank_dirac_1997}. Hence, it may be possible to construct a metric $g$ up to conformal equivalence from the spinor spectrum given that the spinor spectrum is already homeomorphic to $M$. This should not be shocking; Belishev in \cite{belishev_calderon_2003} was able to do this for surfaces $S$ with single boundary component, as he proved that the topologized spectrum is conformally equivalent to $S$ (and in fact was determined by the classical DN operator). It could be that an procedure analogous to Belishev's technique for finding a conformal metric in \cite{belishev_calderon_2003} can be performed with the spinor spectrum.

It could be that we can do better than extracting just the conformal class for dimension $n\geq 3$. As a reminder, the 2-dimensional problem cannot determine more than the conformal class of $g$ since $\Delta$ is conformally invariant in dimension 2. But, if we consider subsurfaces $S$ inside of $M$, we can collect conformal copies of the metric restricted to the surface, vary the over a collection of surfaces passing through a point, and perhaps the combined data from all surfaces passing through each point could produce a metric in the isometry class of $M$. 
\end{enumerate}

\subsubsection{Characters on arbitrary compact $M$}
Aside from the above points, we want the results of this paper to hold true for arbitrary compact $M$, not just compact regions of $\R^n$. We briefly discussed the issue with our proof technique just before the proof of \cref{thm:gelfand}, but the core issue is that our proof hinged on a global power series representation which was valid since $M$ was embedded in $\R^n$. If the power series is only local, then we must, in some sense, understand the restriction of spin characters to local coordinate patches, but this is not understood. 

To view the spin characters from a different perspective, it could be interesting to take $\delta \in \characters(M)$ and consider $\ker \delta$. In the case for a surface $S$, the kernel of a character is in one-to-one correspondence with the set of maximal ideals of the algebra of holomorphic functions. Succinctly, we can match a character $\delta_x$ with the class of holomorphic functions $[f]$ who vanish at the point $x$. The maximal ideals of the space of holomorphic functions are exactly the functions that vanish at just a single point. It is not clear that elements in $\ker \delta$ have this property when the dimension of $M$ exceeds 2. Part of the proof for the 2-dimensional result used by Belishev in \cite{belishev_calderon_2003} follows from \cite[Exercise 26.4, pg. 205]{forster_lectures_1981} which can be proven using sheaves. To that end, it may be useful to think of the space $\monogenics^+(M)$ in the context of sheaves.

On a different note, it could also be useful to identify the $\G$-currents $\contcurrents{\G}{\G}$ with $\G$-valued Radon measures. Additivity of measures over subsets and the regularity of Radon measures may allow for characters to be applied to local power series representations of the monogenic spinor fields. If this is the case, compactness of $M$ would mean that a spin character corresponds to evaluation at finitely many points. As a final step, we could possibly use the fact that $\overline{\monogenics^+(M)}$ separates points to conclude that a character $\delta \in \characters(M)$ is evaluation at only a single $x_\delta \in M$.

\section{Conclusion}

The heart of this paper is to extend the theory Gelfand on commutative Banach algebras of $\C$-valued functions to noncommutative Banach algebras of $\G$-valued functions. In essence, we can find copies of $\C$ as subalgebras of plane spinors $\planespinors \subset \G^+$ and copies of complex holomorphic functions as the monogenic surface spinors $\algebra{\bivector}$. Using the fact we can also locally construct a power series for monogenic spinor fields in terms of monogenic variables $z$ with coefficients in $\G^+$, we define the characters accordingly. Hence, we have a meaningful notion of a spectrum and achieve our main theorem. 

This theory, when restricted to a 2-dimensional surface, yields the Gelfand representation, but allows us to achieve new results in higher dimensions with the same process. More or less, this hinges on the Cauchy integral formula for multivector fields in the same way the Cauchy integral acts as a linchpin in many of the classical complex analysis theorems. Hopefully this entices others to consider the role of the Banach algebra of $\G$-valued functions and the dual space of $\G$-currents.

We expect that more results will follow in the future. For example, we suspect \cref{thm:gelfand} can be generalized to arbitrary compact Riemannian manifolds. If this is to help towards a solution of the Calder\'on problem, this information would all need to be extracted from the boundary and, in particular, from the Dirichlet-to-Neumann operator. Experts in Clifford analysis and elliptic theory may have useful insights into this problem.

\bibliographystyle{unsrt}
\bibliography{clifford_gelfand}

\end{document}